\documentclass{amsart}

\usepackage{enumerate}
\usepackage{enumitem}
\usepackage[utf8]{inputenc}
\usepackage{amsmath}
\usepackage[T1]{fontenc}
\usepackage[british]{babel}
\usepackage{amssymb}
\usepackage{amsthm}
\usepackage{mathtools}
\usepackage[paper=a4paper,left=25mm,right=25mm,top=25mm,bottom=25mm]{geometry}
\usepackage{cite}
\usepackage{hyperref}
\usepackage{cleveref}
\usepackage{bm}

\usepackage{titlesec}
\titleformat{\section}{\normalfont\Large\bfseries}{\thesection}{1em}{}
\titlespacing*{\section}
{0pt}{4.5ex plus 1ex minus .5ex}{2.8ex plus .2ex}

\titleformat{\subsection}{\normalfont\large\bfseries}{\thesubsection}{1em}{}
\titlespacing*{\subsection}
{0pt}{3.5ex plus .8ex minus .2ex}{2.3ex plus .2ex}

\makeatletter
\def\th@plain{%
  \thm@notefont{}%
  \itshape
}
\def\th@definition{%
  \thm@notefont{}%
  \normalfont
}
\makeatother

\newtheorem{theorem}{Theorem}[section]
\newtheorem{corollary}[theorem]{Corollary}
\newtheorem{lemma}[theorem]{Lemma}

\newtheorem{conjecture}[theorem]{Conjecture}

\theoremstyle{definition}
\newtheorem*{remark}{Remark}
\newtheorem*{claim}{Claim}

\newtheorem{definition}[theorem]{Definition}

\crefname{theorem}{Theorem}{Theorems}
\crefname{lemma}{Lemma}{Lemmas}
\crefname{conjecture}{Conjecture}{Conjectures}
\crefname{corollary}{Corollary}{Corollaries}
\crefname{proposition}{Proposition}{Propositions}
\crefname{definition}{Definition}{Definitions}

\title{The Abu-Khzam--Langston Conjecture for Graphs with $\bm{\alpha(G) = 2}$}
\author{Jonathan C.\ Dahlke}
\address{Technische Universität Ilmenau}
\email{jonathan.dahlke@tu-ilmenau.de}

\date{}
\allowdisplaybreaks


\begin{document}

\begin{abstract}
    The Abu-Khzam--Langston conjecture---the weak-immersion analogue of Hadwiger's conjecture and a weak version of an earlier conjecture of Lescure and Meyniel---asserts that every graph $G$ contains a weak immersion of $K_{\chi(G)}$. We prove the conjecture for the class of graphs of independence number two.
    Along the way, we introduce a notion of \emph{cycle-matching colouring} of a graph---a relaxation of edge-colouring in which colour classes induce vertex-disjoint unions of edges and odd cycles---and prove a sharpening of Vizing's theorem in this setting: every multigraph admits a cycle-matching colouring with at most $\Delta(G)$ colours.
\end{abstract}

\maketitle

\section{Introduction}

A graph $H$ is said to be a \emph{(weak) immersion} of a graph $G$ if there is an injection $\iota \colon V(H)\to V(G)$ together with a family of pairwise edge-disjoint paths in $G$, one for each edge $uv\in E(H)$, whose endpoints are $\iota(u)$ and $\iota(v)$. The immersion is a \emph{strong immersion} if, in addition, the internal vertices of these paths avoid the image $\iota(V(H))$. Note that throughout the paper, whenever we use the term immersion, we refer to a weak immersion.

In 1989, Lescure and Meyniel~\cite{LescureMeyniel1989} formulated the following analogue of Hadwiger's famous conjecture~\cite{Hadwiger1943} for strong immersions.

\begin{conjecture}[Lescure, Meyniel \cite{LescureMeyniel1989}]\label[conjecture]{conj:LM-strong}
    Every graph $G$ contains a strong immersion of $K_{\chi(G)}$.
\end{conjecture}

Independently, Abu-Khzam and Langston~\cite{AbuKhzamLangston2003} proposed the corresponding statement for weak immersions, which is the conjecture we focus on in this paper.

\begin{conjecture}[Abu-Khzam, Langston \cite{AbuKhzamLangston2003}]\label[conjecture]{conj:AKL}
    Every graph $G$ contains a weak immersion of $K_{\chi(G)}$.
\end{conjecture}

Since every strong immersion is a weak immersion, \cref{conj:LM-strong} clearly implies \cref{conj:AKL}; both remain open in general. While Hadwiger's conjecture is currently settled only up to $\chi(G)\leq 6$~\cite{RobertsonSeymourThomas1993} and notoriously open beyond, the immersion analogues are a bit better understood: both conjectures are trivial for $\chi(G)\leq 4$ and were verified for $\chi(G)\in\{5,6\}$ by Lescure and Meyniel themselves~\cite{LescureMeyniel1989}; the case $\chi(G) = 7$ was settled by DeVos, Kawarabayashi, Mohar, and Okamura~\cite{DeVosKawarabayashi2010}. For $\chi(G)\geq 8$, both conjectures remain open.

A natural restricted setting for Hadwiger's conjecture, \cref{conj:LM-strong}, and \cref{conj:AKL} is the class of graphs of small independence number. The case $\alpha(G) = 2$ has received particular attention. An initial investigation of \cref{conj:AKL} restricted to this class of graphs was undertaken by Vergara~\cite{Vergara2017}, who, among other things, showed that for graphs with $\alpha(G)\leq 2$ it is equivalent to the assertion that every such graph contains a weak immersion of $K_{\lceil |V(G)|/2\rceil}$ and proposed 

\begin{conjecture}[Vergara \cite{Vergara2017}]\label[conjecture]{conj:vergara}
    Every graph $G$ with $\alpha(G)= 2$ contains a weak immersion of $K_{\chi(G)}$.
\end{conjecture}

A number of partial results have followed. Quiroz~\cite{Quiroz2021} verified \cref{conj:vergara} for $H$-free graphs with $|V(H)|\leq 4$. Bustamante, Quiroz, Stein, and Zamora~\cite{Bustamante2022} obtained immersion-size lower bounds in terms of the independence number for graphs of arbitrary independence number. Botler, Jiménez, Lintzmayer, Pastine, Quiroz, and Sambinelli~\cite{Botler2024} proved a biclique analogue, showing that every graph with $\alpha(G) = 2$ contains an immersion of $K_{\lceil n/2 \rceil, \lceil n/2 \rceil}$; a subsequent shorter proof was proposed by Chen and Deng~\cite{ChenDeng2024}. Most recently, Botler, Fernandes, Lintzmayer, Lopes, Mishra, Netto, and Sambinelli~\cite{Botler2025} proposed a proof of \cref{conj:vergara} for graphs satisfying the maximum-degree bound $\Delta(G) < \tfrac{19}{29}n - 1$ (with $n\geq 11$).

In this paper, we resolve \cref{conj:vergara}.

\begin{theorem}\label{thm:main_theorem}
    Let $G$ be a graph with $\alpha(G)=2$. Then $G$ contains a weak immersion of $K_{\chi(G)}$.
\end{theorem}


\paragraph{\textbf{Notation.}} We use $\chi(G), \alpha(G), \Delta(G), \chi'(G)$ for the chromatic number, independence number, maximum degree, and chromatic index of $G$, respectively. For two vertex sets $A, B\subseteq V(G)$, $E(A, B)$ denotes the set of edges with one endpoint in $A$ and the other in $B$. We write $E(v, A) \coloneq E(\{v\}, A)$, $E(A)\coloneq E(A, A)$, and  $d_G(v) = |E(v, V(G))|$ for the degree of $v$. For an edge set $F\subseteq E(G)$, $G[F]$ is the graph on $V(G)$ with edge set $F$. For a vertex set $A\subseteq V(G)$, $G[A]$ is the graph on $A$ with edge set $E(A)$. A component of $G$ consisting of a single vertex is called trivial. A vertex set $X$ is spanned by an edge set $F$, if every $x\in X$ has $d_{G[F]}(x)>0$. For notation not defined here, we refer to \cite{diestel2025graph}. Throughout, all graphs are finite and loopless; the multigraphs appearing in \cref{sec:vizing} are allowed to have parallel edges.

\section{A Vizing-type theorem for cycle-matching colourings}\label{sec:vizing}

In this section we develop the edge-colouring tool that lies at the heart of our argument. Throughout the section, $G$ may be a multigraph.

\begin{definition}
    A map $f: E(G)\to \mathcal C$ is called an \emph{$r$-bounded regular colouring} if, for each $c\in \mathcal C$, each component of the subgraph $G[f^{-1}(c)]$ is regular of degree at most $r$.

    For $r = 2$, we also call $f$ a \emph{cycle-matching colouring}, since each colour class is then a vertex-disjoint union of isolated edges and odd cycles.

    The associated chromatic index $\chi'_r(G)$ is the smallest number of colours admitting an $r$-bounded regular colouring of $G$.
\end{definition}

For $r=1$ this definition recovers the standard notion of chromatic index, and so $\chi'_1(G) = \chi'(G)$. As every $r$-bounded regular colouring is also $r'$-bounded regular for $r'\geq r$, the parameter $\chi'_r(G)$ is non-increasing in $r$; in particular it is well defined for every $r\in \mathbb N^+$.

By a theorem of Shannon~\cite{shannon1949}, $\chi'_r(G)\leq \chi'(G)\leq \frac{3}{2}\Delta(G)$, where the upper bound can be sharpened to $\Delta(G) + 1$ for simple graphs~\cite{Vizing1964}. For $r\geq 2$, in \cite{cornaznguyen2013}, the authors introduce the notion of \emph{ocm sets}, which are edge sets consisting of vertex disjoint odd cycles and isolated edges, and they prove that every graph $G$ is covered by exactly $\Delta(G)$ ocm sets and no fewer. For them, covered means that an edge is contained in at least two odd cycles of different ocm sets of the covering or is an isolated edge in at least one of the sets. This implies

\begin{theorem}[Cornaz, Nguyen \cite{cornaznguyen2013}]\label{thm:cornaznguyen}
    For every multigraph $G$, there exists an edge set $F\subseteq E(G)$ that spans every vertex of maximum degree in $G$.
\end{theorem}

We obtain by an induction argument provided later in this paper:

\begin{corollary}\label[corollary]{thm:vizing-cycle-matching}
    For every multigraph $G$ and every integer $r\geq 2$,
    \[
        \chi'_r(G) \leq \Delta(G).
    \]
\end{corollary}

We shall give a different proof of \cref{thm:cornaznguyen}, as we will---to give a specialised lemma needed for the proof of \cref{thm:main_theorem}---explore the structure of obstructions to the existence of spanning colour classes.

Note that the bound in \cref{thm:vizing-cycle-matching} is best possible for any $r\geq 2$: for the star $K_{1, s}$, one has $\chi'_r(K_{1, s}) = s$ for every $r\in \mathbb N^+$. The lower bound of $\Delta(G)$ that applies to the covering by ocm sets does not apply, as we do not require an edge contained in an odd cycle to be covered by more than one set in our notion of colouring. Instead, we have $\chi'_r(G)\geq \Delta(G)/r$ (forced by the edges at a vertex of maximum degree). This is tight---it is attained, for example, when $G$ decomposes into $s$ edge-disjoint $r$-regular spanning subgraphs, in which case $\chi'_r(G) = s = \Delta(G)/r$.

\medskip
To prove \cref{thm:vizing-cycle-matching} we use the classical $f$-factor theorem of Lovász~\cite{Lovasz1970} (see also \cite[Theorem~10.2.18]{LovaszPlummer1986} or \cite[Chapter 31]{schrijver2003}; for a detailed write-up of how to get to the formulation used here, see \cite{QuWest2026}).

\begin{definition}
    Let $f\colon V(G)\to \mathbb N$ be a function pointwise bounded by the degree function $d_G$. A subgraph $H\leq G$ is called \emph{$f$-bounded} if its degree function is pointwise bounded by $f$.

    For an ordered pair of disjoint vertex sets $(S, T)$, the \emph{deficiency} is
    \[
        \mathrm{def}(S, T) = f(T) - f(S) + q(S, T) - d_{G-S}(T),
    \]
    where $q(S, T)$ is the number of components $C$ of $G-(S\cup T)$ for which $f(V(C)) + |E(C, T)|$ is odd and we mean $f(T) = \sum_{a\in T}f(a)$. 
\end{definition}

\begin{theorem}[Lovász \cite{Lovasz1970}]\label{thm:f-factor}
    Let $(S, T)$ be a pair of disjoint vertex sets. Then every $f$-bounded subgraph has degree-sum at most $f(V(G)) - \mathrm{def}(S, T)$. Moreover, if $(S, T)$ is a pair of maximum deficiency, there exists an $f$-bounded subgraph $H$ with degree-sum exactly $f(V(G)) - \mathrm{def}(S, T)$. In that case, $H$ is called a maximum $f$-bounded subgraph.
\end{theorem}

We apply this to the special case of doubled graphs, where every edge has even multiplicity and $f\equiv 2$:

\begin{corollary}\label[corollary]{cor:2-factor-properties}
    Let $G$ be a multigraph in which every edge has even multiplicity, and let $f \equiv 2$. Among all pairs $(S, T)$ of disjoint vertex sets attaining the maximum deficiency, choose one that is $\subseteq$-minimal.

    Then there exists a maximum $f$-bounded subgraph $H\leq G$ with the following properties:
    \begin{enumerate}
        \item $H$ is a disjoint union of $2$-cycles and odd cycles;
        \item $\mathrm{def}(S, T) = f(T) - f(S) = 2|T| - 2|S|$;
        \item $T$ is an independent set and $N_G(T) = S$;
        \item $H[S\cup T]$ consists of $2$-cycles together with $|T| - |S|$ isolated vertices, all of which lie in $T$;
        \item for every nonempty $X\subseteq S$, $|N_G(X)\cap T| > |X|$;
        \item every vertex $v\in V(G)$ with $d_G(v) = \Delta(G)$ satisfies $d_H(v) = 2$.
    \end{enumerate}
\end{corollary}

\begin{proof}
    We may choose $H$ to consist only of odd cycles and $2$-cycles as paths can be replaced by $2$-cycles spanning their vertices (possibly missing one vertex due to parity), and even cycles by a spanning set of $2$-cycles, without lowering the number of edges of $H$ and thus the degree-sum. Among all such choices of $H$, fix one covering the maximum number of vertices that have degree $\Delta(G)$ in $G$.

    Since every edge has even multiplicity, every edge-cut of $G$ has even size, and for any vertex set $A\subseteq V(G)$, $f(A) = 2|A|$ is even; hence $q(S, T) = 0$. If some vertex $x\in T$ satisfies $d_{G-S}(x) > 0$, then in fact $d_{G-S}(x) \geq 2$ (by even multiplicity), and
    \begin{align*}
        \mathrm{def}(S, T\setminus\{x\})
        &= f(T\setminus\{x\}) - f(S) - d_{G-S}(T\setminus \{x\}) \\
        &= \bigl(f(T) - 2\bigr) - f(S) - \bigl(d_{G-S}(T) - d_{G-S}(x)\bigr) \\
        &= \mathrm{def}(S, T) + d_{G-S}(x) - 2 \;\geq\; \mathrm{def}(S, T),
    \end{align*}
    contradicting the $\subseteq$-minimality of $(S, T)$. Hence $d_{G-S}(x)= 0$ for every $x\in T$; in particular, $T$ is independent and $N_G(T)\subseteq S$. For the reverse inclusion, $\mathrm{def}(N_G(T), T) = f(T) - f(N_G(T)) \geq f(T) - f(S) = \mathrm{def}(S, T)$, so minimality forces $S = N_G(T)$.

    Suppose now some odd cycle $C$ of $H$ contains a vertex of $S\cup T$. Since $V(C)\cap T$ is independent in $C$, the cycle has more vertices of $S$ than of $T$, so some $S$-vertex has only one $H$-edge into $T$. But then
    \[
        2|S| \;=\; f(T) - \mathrm{def}(S, T) \;\leq\; \sum_{v\in T}d_H(v) \;=\; |E_H(T, S)| \;\leq\; \sum_{y\in S} d_H(y) \;=\; 2|S|,
    \]
    forcing $|E_H(T, S)| = \sum_{y\in S} d_H(y)$, a contradiction. Thus every vertex of $S$ lies in a $2$-cycle of $H$ whose other vertex is in $T$, and the remaining $|T|-|S|$ vertices of $T$ are isolated in $H$.

    Let now $\emptyset\neq X\subseteq S$. Since every vertex of $S$ is paired by a $2$-cycle with a (distinct) vertex of $T$, we have $|N_G(X)\cap T|\geq|X|$. Suppose $|N_G(X)\cap T| = |X|$ and consider $(A, B) = (S\setminus X,\; T\setminus N_G(X))$. Since we removed exactly those vertices of $T$ having a neighbour in $X$, $N_G(B)\subseteq A$, and so $\mathrm{def}(A, B) = 2|B| - 2|A| = \mathrm{def}(S, T)$, contradicting the $\subseteq$-minimality of $(S, T)$.

    Finally, let $X\subseteq T$ be the set of vertices of $T$ of maximum $G$-degree. As $|N_G(Y)\cap S|\geq|Y|$ for every $Y\subseteq X$, otherwise there must be a vertex in $N_G(Y)$ of larger degree, Hall's theorem provides a set of disjoint $2$-cycles covering $X$ (using only edges of $G$, each of multiplicity $\geq 2$). The family of subsets of $T$ covered by disjoint $2$-cycles in $G[S\cup T]$ forms a transversal matroid of rank $|S|$; we can therefore choose a set $F$ of disjoint $2$-cycles covering all of $S\cup X$. If some vertex of $X$ had degree 0, replacing $E(H[S\cup T])$ by $F$ would yield an $f$-bounded subgraph covering strictly more vertices of maximum degree, contradicting our choice of $H$. Vertices of $V(G)\setminus(S\cup T)$ of degree $\Delta(G)$ lie on cycles of $H$ and there their $H$-degree is $2$. This proves property~(6).
\end{proof}

\begin{proof}[Proof of \cref{thm:vizing-cycle-matching}]
    Induction on $\Delta(G)$. The base case $\Delta(G) = 0$ is trivial. For the inductive step, let $H$ be the multigraph obtained from $G$ by doubling every edge, so that every edge of $H$ has even multiplicity. By \cref{cor:2-factor-properties} applied to $H$, there is an $f$-bounded subgraph $H'\leq H$ (with $f\equiv 2$) consisting of disjoint $2$-cycles and odd cycles and covering every vertex of maximum degree in $H$ (equivalently, in $G$). Replacing each $2$-cycle of $H'$ by a single edge yields an edge set $F\subseteq E(G)$ in which every component of $G[F]$ is $1$- or $2$-regular and every vertex of maximum degree in $G$ has positive degree in $G[F]$. Setting $f(F)$ to one colour and applying induction to $G-F$ (whose maximum degree is strictly smaller) extends $f$ to a cycle-matching colouring of $G$ using at most $\Delta(G)$ colours.
\end{proof}

\begin{remark}
    A third route to \cref{thm:cornaznguyen} goes via simple $2$-matchings (for a definition, consult \cite{LovaszPlummer1986})---the colour classes of a cycle-matching colouring are precisely simple $2$-matchings of $G$, and one can show using the machinery of \cite{LovaszPlummer1986} that $G$ admits a simple $2$-matching saturating every vertex of maximum degree. We do not pursue this route here because the additional information furnished by \cref{cor:2-factor-properties}---in particular, point (3) and the structure of $H[S\cup T]$---is what powers the following key lemma.
\end{remark}

\begin{lemma}\label[lemma]{Lem:criticalLemma}
    Let $G$ be a graph and let $\mathcal C$ be a set of colours. For each vertex $x\in V(G)$, let $L_x, M_x, R_x\subseteq \mathcal C$ be disjoint with union $\mathcal C$ (forming a partition where some sets may be empty), and suppose that $|L_x| + |M_x| \geq d_G(x)$ for every $x\in V(G)$.

    Then there exist disjoint edge sets $F, F'\subseteq E(G)$ together with maps
    \[
        \alpha\colon F\to \bigcup_{x\in V(G)}(\{x\}\times M_x),
        \qquad
        \beta\colon F'\to \mathcal C,
        \qquad
        f\colon E(G)\setminus(F\cup F')\to \mathcal C,
    \]
    such that
    \begin{enumerate}
        \item $\alpha$ is injective, and for each $e\in F$, writing $\alpha(e) = (x, c)$, the vertex $x$ is an endpoint of $e$ and is isolated in the graph $G[f^{-1}(c)]$;
        \item for each $c\in \mathcal C$, $\beta^{-1}(c)$ consists of vertex-disjoint length-$2$ paths $u v w$ with $c\in L_u\cap L_v$ and $c\in R_w$; further, $u, v, w$ are isolated in the graph $G[f^{-1}(c)]$.
        \item $f$ is a cycle-matching colouring of $G - (F\cup F')$, and moreover every monochromatic odd cycle of colour $c$ has all of its vertices in $\{x \in V(G): c\in L_x\}$.
    \end{enumerate}
\end{lemma}

\begin{proof}
    We proceed by induction on $|\mathcal C|$, mirroring the proof of \cref{thm:vizing-cycle-matching} using $|\mathcal C|$ instead of $\Delta(G)$ and tracking one additional invariant. During the inductive process, whenever a vertex $x$ is not covered by some colour $c$ at the current step, we say that $x$ is \emph{marked with colour $c$}. We maintain that no two marked vertices are adjacent in $G$ and that every unmarked vertex $x$ satisfies $d_G(x) \leq |L_x| + |M_x|$. No vertex starts marked and the premise $|L_x| + |M_x|\geq d_G(x)$ guarantees the invariant. Further, we maintain that all vertices of $d_G(x) = |\mathcal C|$ will be spanned by the edges removed in that step.

    Fix $c\in \mathcal C$ and let $H$ be the multigraph obtained from $G$ by doubling its edges. By \cref{cor:2-factor-properties}, there is a pair of disjoint vertex sets $(S, T)$ and a subgraph $H'\leq H$ consisting of disjoint $2$-cycles and odd cycles with the listed properties. As before, $H'$ corresponds to an edge set $E'\subseteq E(G)$ whose components are isolated edges and odd cycles. We now mark all vertices $x$ with $d_H(x) = 0$ with the colour $c$, observe that these newly marked vertices are not adjacent to one another.

    Now consider an odd cycle $D\subseteq E'$. If $\forall x\in V(D): c\in L_x$, we keep $E(D)$ in $E'$. Otherwise:
    \begin{itemize}
        \item If $D$ contains a vertex $x$ with $c\in M_x$, pick an edge $xy\in E(D)$, remove it from $E'$, place it in $F$, and set $\alpha(xy) = (x, c)$. We then replace $E(D)$ in $E'$ by a perfect matching of $V(D)\setminus\{x\}$ of edges in $E(D)$ (which exists since $D-x$ is a path with an even number of vertices).
        \item Otherwise, if $D$ contains a vertex $x$ with $c\in R_x$ and $d_G(x) < |\mathcal C|$, we similarly replace $E(D)$ in $E'$ by a perfect matching of $V(D)\setminus\{x\}$ of edges in $E(D)$. We will, in this case, not span $x$ by the edge set removed in this step, which is possible as $d_G(x)<|\mathcal C|$.
        \item Otherwise, every vertex $x$ of $D$ with $c\in R_x$ has $d_G(x) = |\mathcal C|>|L_x| + |M_x|$ and is therefore (by the inductive invariant) marked. Since marked vertices are pairwise non-adjacent and $|V(D)|$ is odd, $D$ must contain a path $uvw$ with $c\in L_u\cap L_v$ and $c\in R_w$. We move the edges $uv, vw$ from $E'$ to $F'$, set $\beta(uv) = \beta(vw) = c$, and replace $E(D)$ in $E'$ by a perfect matching of $V(D)\setminus\{u, v, w\}$.
    \end{itemize}
    Thus, using $\alpha$ and $\beta$, we can ensure that every monochromatic odd cycle in the colour $c$ has all of its vertices in $\{x: c\in L_x\}$. We then set $f(E') = c$ and apply induction to $G' = G - (F \cup F' \cup E')$ with the colour set $\mathcal C\setminus\{c\}$ and, for each vertex $x$, $Q_x' = Q_x\setminus\{c\}$ for $Q\in\{L, M, R\}$. Further note that the only vertices not spanned by the edges assigned to $F, F'$, or $E'$ in this step are marked or have degree less than $|\mathcal C|$.

    Let now $x$ be a vertex with $d_H(x) = 2$. If $x$ is unmarked, then $d_{G'}(x) < d_G(x)$ or $c\in R_x$. In the first case, we observe $d_{G'}(x) \leq d_G(x) - 1 \leq |L_x| + |M_x| - 1 \leq |L_x'| + |M_x'|$ as $c$ can only be in at most one of the sets. In the latter case we have $d_{G'}(x) \leq d_G(x) \leq |L_x| + |M_x| = |L_x'| + |M_x'|$.

    It remains to verify that the vertices marked at the current step (the vertices of $T$ which we leave uncovered by colour $c$) are not adjacent in $G$ to any vertices marked at a previous step. Suppose, for contradiction, that some $x\in T$ has a neighbour $y\in S$ marked with a colour $d\neq c$. At the step where $y$ became marked, we selected an edge set $E'_d$ together with a pair $(S_d, T_d)$ playing the role of $(S, T)$, and $y\in T_d$. Consider the component $D$ of $G[E' \cup E'_d]$ containing $y$. We have $y\in T_d\cap S$. Since $y\in T_d$, $y$ is isolated in $E'_d$; since $y\in S$, $y$ has exactly one incident edge in $E'$, going to some neighbour $y'\in T$. As $y\in T_d$ and $yy'$ lies in $E'\subseteq E(G)$, we have $y'\in N_G(T_d) = S_d$. The same analysis applied to $y'$ shows that $y'$ has exactly one $E'_d$-neighbour in $T_d$ and lies in $S_d\cap T$. Iterating, the component $D$ consists of vertices that alternate between $S_d\cap T$ and $T_d\cap S$, with all vertices $z\in V(D)\setminus\{y\}$ having degree $2$ in $G[E'\cup E'_d]$; only $y$ has degree $1$. This contradicts the fact that the degree sum of every component is even.
\end{proof}

\section{Faithful immersions and proof of the main theorem}\label{sec:immersions}

\begin{definition}
    Let $G$ be a graph and $t\in \mathbb N$. A (\emph{weak}) \emph{$K_t$-immersion} in $G$ consists of a \emph{corner set} $\mathcal K\subseteq V(G)$ of size $t$ together with a family of pairwise edge-disjoint paths in $G$, one for each unordered pair of distinct corners. Unless explicitly indicated, all immersions appearing below are weak.

    Let $\mathfrak A$ be a partition of $V(G)$ into independent sets. A $K_t$-immersion is \emph{$\mathfrak A$-pseudo-faithful} if $|A\cap\mathcal K|\leq 1$ for every $A\in\mathfrak A$. It is \emph{$\mathfrak A$-faithful} if, in addition, for any two corners $u\in A\in \mathfrak A$ and $v\in B\in\mathfrak A$, all edges of the chosen $u$-$v$ path lie in $G[A\cup B]$.
\end{definition}

Throughout the rest of this section, $\mathfrak A$ denotes a $\chi(G)$-colouring of $G$. Since $\alpha(G) = 2$, every class of $\mathfrak A$ has size at most two, and we write
\[
    \mathfrak A_1 = \{A\in\mathfrak A : |A| = 1\}, \qquad \mathfrak A_2 = \mathfrak A\setminus\mathfrak A_1
\]
for the singleton classes and the pair classes, respectively.

\begin{lemma}\label[lemma]{lem:single_point_of_attachment}
    Let $\{u\}, \{v\}\in\mathfrak A_1$ and $A\in\mathfrak A_2$. If $|E(u, A)| = |E(v, A)| = 1$, then $u$ and $v$ share a neighbour in $A$.
\end{lemma}

\begin{proof}
    Otherwise, writing $A = \{a_1, a_2\}$ with $ua_2\notin E(G)$ and $va_1\notin E(G)$, the partition $(\mathfrak A\setminus \{\{u\},\{v\}, A\})\cup \{\{v, a_1\}, \{u, a_2\}\}$ is a valid colouring of $G$ using $\chi(G) - 1$ colours, a contradiction.
\end{proof}

\begin{lemma}\label[lemma]{lem:clique-stuff}
    Let $\{v\}\in \mathfrak A_1$ and set $\mathfrak S = \{A\in\mathfrak A_2 : |E(v, A)| = 1\}$. For any two classes $\{a_1, a_2\}, \{b_1, b_2\}\in \mathfrak S$, if $va_2, vb_2\in E(G)$, then $a_1b_1\in E(G)$.
\end{lemma}

\begin{proof}
    The set $\{v, a_1, b_1\}$ contains an edge by $\alpha(G) = 2$, and the edges $va_1$ and $vb_1$ are absent by the definition of $\mathfrak S$; hence $a_1b_1\in E(G)$.
\end{proof}

\begin{lemma}\label[lemma]{lem:4classes_big_connection}
    Let $\{a_1, a_2\}, \{b_1, b_2\}, \{u\}, \{v\}\in \mathfrak A$ be four distinct classes with $ua_1, vb_1\notin E(G)$. Then $\{u, v, a_2, b_2\}$ induces a $K_4$ in $G$; in particular, $a_2b_2\in E(G)$.
\end{lemma}

\begin{proof}
    Since $\{u, v\}$ cannot be a colour class, $uv\in E(G)$. By $\alpha(G) = 2$, $\{v, b_1, b_2\}$ contains an edge, and the only candidate is $vb_2$; similarly $ua_2\in E(G)$.

    If $va_2\notin E(G)$, then $\{v, a_2\}, \{u, a_1\}$ would be a valid recolouring of three colour classes by two, contradicting the optimality of $\mathfrak A$. The same argument gives $ub_2\in E(G)$.

    Finally, if $a_2b_2\notin E(G)$, then $\{u, a_1\}, \{v, b_1\}, \{a_2, b_2\}$ would be a valid recolouring of four colour classes by three, again a contradiction.
\end{proof}

\begin{lemma}\label[lemma]{lem:faithful_immersion}
    Suppose that for every $A\in \mathfrak A_2$ there is some $\{v_A\}\in \mathfrak A_1$ with $|E(v_A, A)| = 1$. Then $G$ admits an $\mathfrak A$-faithful $K_{\chi(G)}$-immersion.
\end{lemma}

\begin{proof}
    By \cref{lem:single_point_of_attachment}, for every $A\in\mathfrak A_2$ we may label $A = \{p_A, c_A\}$ so that, whenever $\{v\}\in\mathfrak A_1$ and $|E(v, A)| = 1$, we have $vp_A\notin E(G)$ (and consequently $vc_A\in E(G)$). Take
    \[
        \mathcal K = \bigcup\mathfrak A_1 \;\cup\; \{c_A : A\in\mathfrak A_2\}
    \]
    as our corner set, so that $|\mathcal K| = |\mathfrak A| = \chi(G)$.

    The set $\bigcup\mathfrak A_1$ is a clique in $G$ ($\{u, w\}$ cannot be a colour class for distinct $\{u\}, \{w\}\in\mathfrak A_1$). By the labelling, $ac_A\in E(G)$ for every $\{a\}\in\mathfrak A_1$ and every $A\in\mathfrak A_2$. Thus every pair of corners involving at least one singleton corner is realised by a direct edge.

    Now suppose $A, B\in\mathfrak A_2$ are such that $c_Ac_B\notin E(G)$. By hypothesis there exist $\{u\}, \{v\}\in\mathfrak A_1$ with $|E(u, A)| = |E(v, B)| = 1$; by the labelling, $up_A\notin E(G)$ and $vp_B\notin E(G)$. If $u\neq v$, then \cref{lem:4classes_big_connection} applied to $\{u\}, \{v\}, A, B$ would force $c_Ac_B\in E(G)$. Hence $u = v$, and \cref{lem:clique-stuff} (applied at $u$) gives $p_Ap_B\in E(G)$. By $\alpha(G) = 2$, $c_A$ has at least one neighbour in $B$ and $c_B$ at least one in $A$; combined with $c_Ac_B\notin E(G)$, this forces $p_Ac_B, c_Ap_B\in E(G)$. We then use the length-$3$ path $c_A p_B p_A c_B$ as the immersion path between $c_A$ and $c_B$.

    All immersion paths are either direct corner-to-corner edges, or length-$3$ paths whose edges lie in $G[A\cup B]$ for some $A, B\in\mathfrak A_2$; in particular the immersion is $\mathfrak A$-faithful. Edge-disjointness is immediate from the construction: distinct length-$3$ paths $c_A p_B p_A c_B$ use the unique edge $p_Ap_B$ between the relevant non-corner vertices and therefore cannot collide with one another or with any direct edge.
\end{proof}

\begin{proof}[Proof of \cref{thm:main_theorem}]
    We proceed by induction on $|V(G)|$. Among all $\chi(G)$-colourings $\mathfrak A$ of $G$, choose one for which the set
    \[
        \mathfrak X = \bigl\{A\in\mathfrak A_2 \,\bigm|\, \exists\,\{v\}\in\mathfrak A_1 :\, |E(v, A)| = 1 \bigr\}
    \]
    is as large as possible.

    \medskip
    \noindent\textit{Easy cases.}
    If $\mathfrak A_1 = \emptyset$, then every class of $\mathfrak A$ has size two; in particular $|V(G)|$ is even and $\chi(G) = |V(G)|/2$. Removing any vertex $v$ from $G$, we have $\chi(G-v)\geq \lceil (|V(G)|-1)/2\rceil = |V(G)|/2 = \chi(G)$, so $\chi(G-v) = \chi(G)$. By induction, $G-v$ contains an immersion of $K_{\chi(G)}$, and so does $G$.

    If $\mathfrak A_1\neq\emptyset$ but $\mathfrak X = \emptyset$, then every singleton vertex of $\mathfrak A_1$ has only ``full'' connections into each class of $\mathfrak A_2$ (that is, $|E(v, A)|\neq 1$ for every $\{v\}\in\mathfrak A_1$ and every $A\in\mathfrak A_2$). This implies $d_G(v) = |V(G)| - 1$. Removing all singleton vertices from $G$ yields a graph $G'$ with $\chi(G') = \chi(G) - |\mathfrak A_1|$ and $\alpha(G')\leq 2$; by induction, $G'$ contains an immersion of $K_{\chi(G')}$, and the singleton vertices, each adjacent to every other vertex of $G$, extend it to a $K_{\chi(G)}$-immersion of $G$.

    Henceforth, we assume $\mathfrak A_1\neq\emptyset$ and $\mathfrak X\neq\emptyset$.

    \medskip
    \noindent\textit{Setup.}
    Set $\mathfrak Y = \mathfrak A_2\setminus\mathfrak X$. By the inductive hypothesis applied to $G[\bigcup\mathfrak Y]$, there is a $K_{|\mathfrak Y|}$-immersion with corner set $\mathcal C_\mathfrak Y$; by \cref{lem:faithful_immersion} applied to $G[V(G)\setminus\bigcup\mathfrak Y]$ (whose colouring induced by $\mathfrak A$ satisfies the hypothesis of the lemma by the definition of $\mathfrak X$), there is an $\mathfrak A$-faithful $K_{\chi(G)-|\mathfrak Y|}$-immersion with corner set $\mathcal C_\mathfrak X$. For each $A\in\mathfrak X$, write $A = \{p_A, c_A\}$ with $c_A$ the corner. Our task is to merge these two immersions into a $K_{\chi(G)}$-immersion of $G$ by routing edge-disjoint paths between $\mathcal C_\mathfrak Y$ and $\mathcal C_\mathfrak X$. Note that all vertices in $\bigcup \mathfrak A_1$ are connected to all vertices from $\mathcal C_\mathfrak Y$ and we use these edges as paths in our immersion.

    Partition $\mathfrak X$ into sets $(\mathfrak X_v)_{\{v\}\in\mathfrak A_1}$ (some possibly empty) such that every $A\in\mathfrak X_v$ satisfies $|E(v, A)| = 1$. We will connect the corners from $\bigcup\mathfrak X_v$ with those in $\mathcal C_\mathfrak Y$ partition class by partition class. As every corner in $\mathcal C_\mathfrak X$ is contained in one of these sets, by proceeding this way, we do connect all corners with one another.

    Fix $\{v\}\in\mathfrak A_1$ and consider some $X\in\mathfrak X_v$, say $X = \{p_X, c_X\}$. The maximality of $|\mathfrak X|$ implies
    \begin{equation}\label{eq:maximality}
        \bigl|\{Y\in\mathfrak Y : |E(c_X, Y)| = 1\}\bigr| \;\leq\; \bigl|\{A\in\mathfrak X_v : |E(c_X, A)| = 2\}\bigr|,
    \end{equation}
    since otherwise swapping $c_X$ for $v$ in $X$ would yield another $\chi(G)$-colouring with a strictly larger $\mathfrak X$: after the swap, $\mathfrak X_{c_X}$ can be chosen as $\mathfrak X_v$ with the set on the LHS of the inequality added, and the elements of the set on the RHS moved to the new $\mathfrak Y$ or some other $\mathfrak X_{v'}$, if necessary.

    \medskip
    \noindent\textit{Pair types and an auxiliary digraph.}
    For $A\in\mathfrak X_v$ and $y\in\mathcal C_\mathfrak Y$, call the pair $(A, y)$
    \begin{itemize}
        \item \emph{good} if $c_A y\in E(G)$ and $p_A y\in E(G)$;
        \item \emph{semi-deficient} if $c_A y\in E(G)$ and $p_A y\notin E(G)$;
        \item \emph{deficient} if $c_A y\notin E(G)$ and $p_A y\in E(G)$.
    \end{itemize}
    Since $\alpha(G) = 2$ forces at least one edge between $y$ and $A$, every pair falls into exactly one of these three categories.

    For each $A\in\mathfrak X_v$ define $
        L_A = \{y\in\mathcal C_\mathfrak Y : (A, y) \text{ deficient}\}$, $ M_A = \{y\in\mathcal C_\mathfrak Y :\allowbreak (A, y) $ semi-deficient$\}$, and $ R_A = \{y\in\mathcal C_\mathfrak Y :\allowbreak (A, y) \text{ good}\}$,
    so that $\mathcal C_\mathfrak Y = L_A\cup M_A\cup R_A$.

    Define a digraph $D_v$ on the vertex set $\mathfrak X_v\cup\mathfrak Y$ with arc set
    \begin{align*}
        E(D_v) = &\bigl\{AB : A, B\in\mathfrak X_v,\; |E(c_A, B)| = 2\bigr\}\;\cup\;\\&\bigl\{AY : A\in\mathfrak X_v,\;Y\in\mathfrak Y,\; \exists\, y\in Y \setminus\mathcal C_\mathfrak Y: |E(y, A)| = 2\bigr\}.
    \end{align*}
    Each outgoing arc from $A$ corresponds to a length-$2$ ``bridge'' joining $p_A$ to $c_A$ through a non-corner vertex: for an arc $AB$ with $B\in\mathfrak X_v$, the bridge is $p_A p_B c_A$ (the edges exists since $|E(c_A, B)| = 2$ and by \cref{lem:clique-stuff}); for an arc $AY$ with $Y\in\mathfrak Y$, the bridge is $p_A y c_A$ for the $y\in Y\setminus\mathcal C_\mathfrak Y$ with $|E(y, A)| = 2$. Note that these bridges do not use edges from $G[\bigcup\mathfrak Y]$ and are edge-disjoint from one another except if both $AB, BA\in E(D_v)$, in which case the edge $p_Ap_B$ is used twice. Further, none of the edges are used in the faithful immersion on $G[V(G)\setminus\bigcup\mathfrak Y]$, as $|E(c_A, B)| = 2$ implies that this immersion uses the edge $c_Ac_B$ directly and the faithful property ensures that no other edges between $A$ and $B$ are used.

    \begin{claim}\label{claim:digraph-degree}
        For every $A\in\mathfrak X_v$, the out-degree of $A$ in $D_v$ satisfies $d^+_{D_v}(A)\geq |L_A| + |M_A|$.
    \end{claim}

    \begin{proof}
        Fix $A\in\mathfrak X_v$. For $y\in\mathcal C_\mathfrak Y$, let $Y\in\mathfrak Y$ denote the class containing $y$. There are two cases according to whether $Y\cap\mathcal C_\mathfrak Y = \{y\}$ or $|Y\cap\mathcal C_\mathfrak Y| = 2$.

        \emph{Case 1: $Y\cap\mathcal C_\mathfrak Y = \{y\}$.}
        Here $y$ is the unique corner of its class, and the other vertex $y^\star$ of $Y$ is not a corner. If $y\in L_A$ (so $c_Ay\notin E(G)$), then $|E(c_A, Y)|= 1$; if $y\in M_A$, then we have $p_Ay\not\in E(G)$, which implies $p_Ay^\star\in E(G)$. Then, either $|E(c_A, Y)| = 1$, or $|E(y^\star, A)| = 2$. If $|E(c_A, Y)| = 1$, by \eqref{eq:maximality} applied to $A$, each such ``deficiency'' on the $\mathfrak Y$-side is matched by some $A'\in\mathfrak X_v$ with $|E(c_A, A')| = 2$, giving an outgoing arc from $A$. In the case $|E(y^\star, A)|=2$, $AY$ is an arc in $E(D_v)$.

        \emph{Case 2:} $Y\cap\mathcal C_\mathfrak Y = \{y, y'\}$ with $y\neq y'$.
        Choose, for each such $Y\in\mathfrak Y$ contributing two corners, a class $Y'\in\mathfrak Y$ with $Y'\cap\mathcal C_\mathfrak Y = \emptyset$. (Such classes exist in sufficient number: $\mathfrak Y$ is a $\chi(G[\bigcup\mathfrak Y])$-colouring of it, so a $K_{|\mathfrak Y|}$-immersion uses precisely $|\mathfrak Y|$ corners from $|\bigcup\mathfrak Y| = 2|\mathfrak Y|$ vertices, leaving precisely the same number of classes with no corners as classes with two corners; a careful pairing gives a class $Y'$ for each $Y$.) Observe that $Y'$ will necessarily produce an arc for $A$. If $|E(c_A, Y')| = 2$, then at least one vertex in $Y'$ will have all of $A$ as neighbours and then $AY'\in E(D)$. Otherwise, \eqref{eq:maximality} yields an outgoing arc. As for $Y$, if $|E(c_A, Y)| = 2$, then one element of $Y$ must form a good pair with $A$ and hence, from both corners in $Y$, we only need one outgoing arc, which is provided by $Y'$. Otherwise,  we have $|E(c_A, Y)| = 1$ and \eqref{eq:maximality} provides the required arc.

        In either case, distinct $y\in L_A\cup M_A$ end up producing distinct outgoing arcs from $A$. This yields $d^+_{D_v}(A)\geq |L_A| + |M_A|$.
    \end{proof}

    \medskip
    \noindent\textit{Assigning paths.}
    By the previous claim, we may select a sub-digraph $D\leq D_v$ with $V(D) = V(D_v)$ and $d^+_D(A) = |L_A| + |M_A|$ for every $A\in\mathfrak X_v$.

    We will assign a colour $g(e)\in\mathcal C_\mathfrak Y$ to a subset of  arcs $e\in E(D)$ in such a way that for every $A$, $g$ realises an injection from  $L_A$ to the outgoing arcs of $A$, so that distinct corners $y\in L_A$ receive disjoint sets of incident immersion edges connecting $y$ to $c_A$. Concretely, an arc $AB$ assigned the colour $y\in L_A$ encodes the immersion path that starts at $y$, follows either the edge $y p_A$ or---if permissible---$y p_B$, and then uses the bridge $p_A\to c_A$ described by the arc, or the edge $p_Bc_A$ directly.

    The only potential clash among the naive bridges mentioned earlier occurs when both $AB$ and $BA$ are present in $D$ with $A, B\in\mathfrak X_v$: their bridges both use the edge $p_A p_B$. To resolve this, consider the auxiliary undirected graph
    \[
        H = \bigl(\mathfrak X_v,\; \{AB : AB,\, BA \in E(D)\}\bigr).
    \]
    Since $d^+_D(A) = |L_A| + |M_A|$, every vertex of $H$ satisfies $d_H(A)\leq |L_A| + |M_A|$, so we may apply \cref{Lem:criticalLemma} to $H$ with the colour set $\mathcal C_\mathfrak Y$ and the partitions $L_A, M_A, R_A$. This produces $F, F'\subseteq E(H)$ and maps $\alpha\colon F\to \bigcup\{A\}\times M_A$, $\beta\colon F'\to \mathcal C_\mathfrak Y$, and a cycle-matching colouring $f$ of $H - (F\cup F')$.

    We translate these structures back into a global assignment $g$ as follows.
    \begin{itemize}
        \item If $AB\in F$ with $\alpha(AB) = (A, y)$, then we remove the arc $AB$ from $D$. As $y\in M_A$, $c_Ay\in E(G)$ and we can directly connect these corners. This removes the necessity of using $p_Ap_B$ twice, regardless of the assignment of $g(BA)$.
        \item If $\{AB, BC\}\subseteq F'$ form a length-$2$ path with $\beta(AB) = \beta(BC) = y\in L_A\cap L_B\cap R_C$ (in particular $yp_A, yp_B\in E(G)$ and $yp_C, yc_C\in E(G)$), assign $g(AB) = g(BC) = y$ but use the modified bridges $y p_C c_B$ and $y p_B c_A$, avoiding the edges $p_Ap_B$ and $p_Bp_C$.
        \item If $D$ is a monochromatic odd cycle of colour $y$ in $f$ with vertex set $V(D)= \{A_i:i\in \mathbb Z_l\}$, then $y\in L_{A_i}$ for every $i\in\mathbb Z_l$ by \cref{Lem:criticalLemma}; assign $g(A_iA_{i+1}) = y$ and use the bridges $y p_{A_i} c_{A_{i+1}}$, avoiding every edge $p_{A_i}p_{A_{i+1}}$.
        \item If $AB$ is an isolated edge in $f$ of colour $y$: if $y\in R_A\cap R_B$ and $g(AB) = z$ (wlog), we use the bridge $z p_A y p_Bc_A$ to connect $A$ and $z$, whereas connecting $g(BA)$ and $B$ uses the naive bridge $g(BA)p_Bp_Ac_B$. If $y\in M_A$, we drop $AB$; symmetrically for $M_B$. If $y\in L_A\cap L_B$, we colour $g(AB) = g(BA) = y$ and treat it as a $2$-cycle. If $y\in L_A$ and $y\in R_B$, we colour $g(AB) = y$ and use the path $yp_Bp_A$.
    \end{itemize}
    
    \cref{Lem:criticalLemma} guarantees that the deletions imposed by $\alpha$ and $f$ remove at most $|M_A|$ outgoing arcs at each $A$, so we always retain enough out-arcs to cover $L_A$. We now have assigned or thrown away one arc of each double arc of $D$. For each $y\in L_A$ not assigned yet, we choose an arbitrary outgoing arc to assign $y$ to and use the edge $yp_A$ and the naive path between $p_A$ and $c_A$ given by the arc.

    To show that all paths are edge-disjoint, consider the following: All edges used to form non-trivial paths are of the form $yp_A$, $p_Ap_B$ or $p_Bc_A$ for $A, B\in \mathfrak X_v, y\in \mathcal C_\mathfrak Y$. By construction, we avoid using edges of the form $p_Ap_B$ twice. The edge $p_Bc_A$ is used exactly in the path that connects $A$ to $g(AB)$ and therefore can not be used twice either.

    Assume that $yp_A$ is used twice in the construction. If $(A, y)$ is good, then $yp_A$ can only be used if there exists a $B$ with $BA\in F', \beta(BA) = y$ or a $C$ with $f(CA) = y$. As these are mutually exclusive ($A$ either has degree one in $G[f^{-1}(c)]$ or is isolated there and has degree one in $G[\beta^{-1}(c)]$), this can not be. Hence, $(A, y)$ is not good. It can neither be semi-deficient, as then $yp_A\not\in E(G)$, so $(A, y)$ is deficient. In this case, $yp_A$ is used either in the path connecting $y$ to $c_A$ or there is exactly one neighbour $B$ of $A$, for which $yp_A$ is used to connect $y$ to $c_B$. By the construction above, whenever the latter is the case, $A$ has exactly one neighbour $C$ such that $yp_Cc_A$ is the path connecting $y$ to $c_A$. But the former can, by definition, only happen once. Hence the edges of the type $yp_A$ are also used at most once.

    This concludes matching the corners in $\bigcup\mathfrak X_v$ to all vertices of $\mathcal C_\mathfrak Y$. Doing this for each $v\in \mathfrak A_1$ yields paths between every corner in $\mathcal C_\mathfrak X$ and $\mathcal C_\mathfrak Y$ and completes the $K_{\chi(G)}$-immersion.
\end{proof}

\section*{Acknowledgements}

I would like to express my sincere thanks to my advisor, Matthias Kriesell, for his guidance, encouragement, and many helpful conversations. 
 
\bibliographystyle{plain}
\bibliography{references}

\end{document}